\documentclass[12pt,reqno]{amsart}

\usepackage{amssymb,amsthm}

\advance\hoffset by -0.5truein
\let\?=\relax

\def\Cur{\mathop{\fam 0 Cur}\nolimits}
\def\Vir{\mathop{\fam 0 Vir}\nolimits}
\def\gr{\mathop{\fam 0 gr}\nolimits}
\def\Conf{\mathop{\fam 0 Conf}\nolimits}
\def\ComConf{\mathop{\fam 0 ComConf}\nolimits}
\def\End{\mathop{\fam 0 End}\nolimits}

\def\Cend{\mathop {\fam0 Cend}\nolimits}
\def\oo#1{\mathrel {{\circ}_{(#1)}}}
\def\ool#1{\mathrel {{\circ}_{#1}}}

\def\bol#1{\mathrel {{\bullet}_{#1}}}

\newtheorem{theorem}{Theorem}
\newtheorem{proposition}{Proposition}
\newtheorem{corollary}{Corollary}

\theoremstyle{definition}

\newtheorem{definition}{Definition}
\newtheorem{example}{Example}
\newtheorem{remark}{Remark}

\title{Universal enveloping Poisson conformal algebras}
\author{P. S. Kolesnikov}
\address{Sobolev Institute of Mathematics}
\email{pavelsk@math.nsc.ru}

\begin{document}

\begin{abstract}
Lie conformal algebras are useful tools for studying vertex operator algebras and their representations. 
In this paper, we establish close relations between Poisson conformal algebras and 
representations of Lie conformal algebras. We also calculate explicitly Poisson conformal brackets 
on the associated graded conformal algebras of universal associative conformal envelopes of Virasoro 
conformal algebra and Neveu--Schwartz conformal superalgebra. 

\end{abstract}

\maketitle
 
\section{Introduction}

This work was inspired by the following observation. 
Suppose $V$ is a Poisson algebra with operations $x\cdot y$ 
and $[x,y]$ over a field $\Bbbk $. Denote by $\mathfrak g$ the underlying Lie 
algebra structure on $V$ relative to the operation $[\cdot,\cdot]$. 
For a formal variable $\lambda$, consider the following operation $(\cdot \ool\lambda \cdot)$:
\[
 (x\ool{\lambda } u) = [x,u] + \lambda x\cdot u, \quad x,u\in V.
\]
It is straightforward to compute that 
\[
 x\ool{\lambda} (y\ool{\mu }u) - y\ool{\mu} (x\ool{\lambda } u) = [x,y]\ool{\lambda +\mu} u
\]
(see Proposition~\ref{prop:Poisson-LieMod} below for more general computation).
The relation obtained is known as the conformal Jacobi identity \cite{K1} for a conformal module over 
the current Lie conformal algebra $\Cur\mathfrak g$.

In this paper, we study conformal Poisson algebras. They turn to be closely related to
representations of Lie conformal algebras as well as to Gel'fand--Dorfman structures introduced 
in \cite{GD1983}. The latter are known to be in one-to-one correspondence with certain class 
of Lie conformal algebras \cite{Xu1999}. A series of examples of Poisson conformal algebras 
is given by associated graded conformal algebras $\gr U$ of universal associative conformal envelopes of Lie conformal algebras
corresponding to various locality bounds.
We establish 
explicit expressions for the conformal Poisson brackets on $\gr U$ for universal envelopes 
of the Virasoro conformal algebra $\Vir $ for $N=2,3$.
An interesting intermediate example appears as the even part of a universal associative envelope 
of the Neveu--Schwartz conformal superalgebra~$K_1$.

A universal and effective tool for investigations related to universal envelopes 
is the Gr\"obner--Shirshov bases (GSB) theory. 
In Section \ref{sec:GSB}, we present an approach to the calculation of GSBs for associative conformal algebras
based on the GSB theory for modules over ordinary associative algebras. 
Section \ref{sec:exmp} contains two examples: we compute GSBs for two particular universal envelopes 
for the Virasoro conformal algebra and Neveu--Schwartz conformal superalgebra.
As an application, we calculate explicitly the structure of three Poisson conformal envelopes 
$PV_2$, $PV_3$, and $PK_{10}$ 
of the Virasoro conformal algebra in Section~\ref{sec:Brackets}.

\section{Conformal algebras: preliminaries}

In this section, we state definitions and examples of conformal algebras following \cite{K1}.
Throughout the paper, $\Bbbk $ is a field of characteristic zero,
$H=\Bbbk [\partial ]$ is the algebra of polynomials,
$\mathbb Z_+$ 
is the set of non-negative integers. We will use common notation $x^{(s)}$ for $\frac{1}{s!} x^s$, $s\in \mathbb  Z_+$.

A {\em Lie conformal algebra}  $L$ is an $H$-module equipped with 
a family of bilinear operations 
$[\cdot \oo{n} \cdot ]$, $n\in \mathbb Z_+$, such that 
for every $x,y\in L$
\begin{equation}\label{eq:Locality}
 [x\ool\lambda y]:=\sum\limits_{s\ge 0} \lambda ^{(s)} [x\oo{s} y] \in L[\lambda ] 
\end{equation}
where $ L[\lambda ]$ stands for the space  of polynomials over $L$
in a formal variable $\lambda  $,
\begin{equation}
\begin{gathered}\label{eq:3/2-linear}
 [\partial x \oo{n} y ] = -n [x\oo{n-1} y], \\
 [x\oo{n} \partial y] = \partial [x\oo{n} y] + n[x\oo{n-1} y],
\end{gathered}
\end{equation}
\begin{equation}\label{eq:ConfAComm}
 [x\oo{n} y ] = -\sum\limits_{s\ge 0} (-1)^{n+s} \partial ^{(s)} [y\oo{n+s} x]
\end{equation}
for all $x,y \in L$, $n\in \mathbb Z_+$, 
and
\begin{equation}\label{eq:ConfJacobi}
 [x\oo{n} [y\oo{m} z]] - [y\oo{m} [x\oo{n} z]] = \sum\limits_{s\ge 0} \binom {n}{s} [[x\oo{s} y]\oo{n+m+s} z]
\end{equation}
for all $x,y,z\in L$, $n,m\in \mathbb Z_+$.

Condition \eqref{eq:Locality} states that for every pair $x,y\in L$ there exist only a finite number 
of $n\in \mathbb Z_+$ such that  $[x\oo{n} y] \ne 0$. In particular, one may determine {\em locality function}
$N_L:L\times L\to \mathbb Z_+$ in the following way: $N(x,y)$ is the minimal $n\in \mathbb Z_+$ such that 
$[x\oo{m} y]=0$ for all $m\ge n$.

An {\em associative conformal algebra} $C$ is an $H$-module  equipped with 
a series of bilinear operations $(\cdot \oo{n} \cdot )$, $n\in \mathbb Z_+$, 
such that the analogues of \eqref{eq:Locality}, \eqref{eq:3/2-linear} hold and
\begin{equation}\label{eq:ConfAss}
 (x\oo{n} (y\oo{m} z)) = \sum\limits_{s\ge 0} \binom {n}{s} ((x\oo{s} y)\oo{n+m+s} z)
\end{equation}
for all $x,y,z \in C$, $n,m\in \mathbb Z_+$.

It is convenient to write the axioms of Lie and associative conformal algebras in terms 
of generating functions ($\lambda$-products) given by the expression \eqref{eq:Locality}. 
For example, \eqref{eq:3/2-linear} is equivalent to 
\begin{equation}\label{eq:3/2-linearLambda}
 [\partial x\ool{\lambda } y] = -\lambda [x\ool{\lambda } y],
 \quad 
 [x\ool{\lambda } \partial y] = (\lambda +\partial ) [x\ool{\lambda } y],
\end{equation}
\eqref{eq:ConfJacobi} and \eqref{eq:ConfAss} are equivalent to 
\begin{equation}\label{eq:ConfJacobiLambda}
 [x\ool{\lambda } [y\ool{\mu } z]] - [y\ool{\mu}[x\ool{\lambda }z]] = [[x\ool{\lambda }y]\ool{\lambda + \mu }z]
\end{equation}
and 
\begin{equation}\label{eq:ConfAssLambda}
 (x\ool{\lambda } (y\ool{\mu } z)) = ((x\ool{\lambda }y)\ool{\lambda + \mu }z),
\end{equation}
respectively,
where $\lambda $ and $\mu $ are independent commuting variables.

The expression in the right-hand side of \eqref{eq:ConfAComm} is equal to the coefficient at $\lambda^{(n)}$ 
in the expression $[y\ool{-\partial -\lambda } x]$. 
Therefore, \eqref{eq:ConfAComm} is equivalent to 
\[
 [x\ool{\lambda } y] = - [y\ool{-\partial - \lambda } x].
\]

\begin{example}\label{exm:Current}
Let $A$ be a Lie (associative) algebra.
Consider the free $H$-module $C=H\otimes A$.
Define 
\[
 [a\ool{\lambda }b] = [a,b]
\]
for $a,b\in A$, and expand the operation $[\cdot\ool\lambda \cdot]$ to the entire $C$
by \eqref{eq:3/2-linearLambda}.
We obtain Lie (associative) conformal algebra structure called
{\em current} conformal algebra; it is denoted  by $\Cur A$.
\end{example}

\begin{example}\label{exm:Virasoro}
Consider 
1-generated free $H$-module $\Vir=Hv$.
Define 
\begin{equation}\label{eq:VirProduct}
 [v\ool{\lambda }v] = (\partial + 2\lambda )v
\end{equation}
and expand the operation $[\cdot\ool\lambda \cdot]$ to the entire $H$-module
by \eqref{eq:3/2-linearLambda}.
We obtain Lie conformal algebra structure called {\em Virasoro} conformal algebra.
\end{example}

A general class of examples of Lie conformal algebras (quadratic conformal algebras) involving 
current and Virasoro conformal algebras is mentioned in Section \ref{sec:ConfPoisson}, see also \cite{Xu1999}.

\begin{example}\label{exm:Cend}
 Let $A$ be an associative algebra.
 Then the $H$-module
 $C= \Bbbk[\partial,x ] \otimes A\simeq H\otimes A[x]$ equipped with 
 operations
 \[
  (f(\partial, x)\ool{\lambda } g(\partial , x)) = f(-\lambda , x)g(\partial+\lambda, x+\lambda ),\quad f,g\in C,
 \]
 is an associative conformal algebra. If $A=M_n(\Bbbk )$ then the algebra constructed in this way 
 is denoted $\Cend_n$ \cite{K1}. 
\end{example}

As in the world of ordinary algebras, an associative conformal algebra $C$ turns into a Lie one (denoted by $C^{(-)}$)
relative to new operations 
\begin{equation}\label{eq:ConfCommutator}
 [x\oo{n} y] = (x\oo{n} y) - \{y\oo{n} x\}, \quad n\in \mathbb Z_+,
\end{equation}
where 
\begin{equation}\label{eq:Curly}
\{y\oo{n} x\} = \sum\limits_{s\ge 0} (-1)^{n+s} \partial ^{(s)} (y\oo{n+s} x). 
\end{equation}
However, 
there exist Lie conformal algebras that cannot be embedded into associative ones in this way \cite{Roit2000}.

Suppose $V$ is an $H$-module. A {\em conformal linear transformation} $\varphi $ is a rule that turns every $v\in V$ 
into a polynomial $\varphi\ool\lambda v\in V[\lambda ]$ in such a way that 
$\varphi \ool\lambda \partial v = (\partial +\lambda)(\varphi\ool\lambda v)$.
The set of all conformal linear transformations of $V$ is denoted $\Cend V$. The space $\Cend V$ has a natural 
structure of an $H$-module, and there is a $\lambda $-product 
$(\varphi \ool\lambda \psi )\in \Cend V[[\lambda ]]$ given by the rule
\[
 (\varphi \ool\lambda \psi )\ool\mu v = \varphi \ool\lambda (\psi \ool{\mu-\lambda}  v). 
\]
If $V$ is a finitely generated $H$-module then $(\varphi \ool\lambda \psi )$ is a polynomial in $\lambda $ and thus 
$\Cend V$ is an associative conformal algebra. For a free $H$-module $V$ of rank $n$ $\Cend V$ is isomorphic to $\Cend_n$
from Example \ref{exm:Cend}.

Assume $C$ is an associative conformal algebra. 
 An $H$-module $V$ is said to be a {\em conformal module} 
over $C$ if equipped with an $H$-linear map $\rho : C\to \Cend V$ 
preserving the $\lambda$-product. Alternatively, there should exist a
family of bilinear maps $(\cdot \oo{n}\cdot ): C\times V \to V$, $n\in \mathbb Z_+$, 
such that the analogues of \eqref{eq:Locality}, \eqref{eq:3/2-linear}, and \eqref{eq:ConfAss} hold. 
In a similar way, a conformal representation of a Lie conformal algebra $L$ is defined 
as an $H$-linear map $\rho: L\to (\Cend V)^{(-)}$ preserving the operation $[\cdot \ool\lambda \cdot ]$.

Conformal algebra (or a module over a conformal algebra) is said to be {\em finite} if it is finitely generated as an $H$-module. 
There is an open problem whether every finite Lie conformal algebra may be embedded into an associative one.
In \cite{Kol2016}, it was shown that if a finite Lie conformal algebra $L$ 
is a torsion-free $H$-module and satisfies Levi condition (i.e., if its solvable radical splits)
then $L$ has a finite faithful representation, thus may be embedded into an associative conformal algebra $\Cend_n$ for 
an appropriate~$n$.

\section{Poisson conformal algebras and universal enveloping conformal algebras}\label{sec:ConfPoisson}

A more conceptual and general approach to the theory of conformal algebras was proposed in \cite{BDK2001}. 
Consider $H$ as a Hopf algebra generated by primitive element $\partial $. Then the class $\mathcal M^*(H)$ of 
$H$-modules is a pseudo-tensor category in the sense of \cite{BD} relative to a natural composition rule. 
Then Lie (or associative) conformal algebra may be defined as a morphism from the operad Lie (or As) to 
$\mathcal M^*(H)$. Generator of the corresponding operad maps to an $H$-bilinear product ({\em pseudo-product})
\[
 *: C\otimes C \to (H\otimes H)\otimes _H C,\quad C\in \mathcal M^*(H), 
\]
where 
\[
 x*y = \sum\limits_{s\ge 0} ((-\partial)^{(s)}\otimes 1)\otimes _H (x\oo{s} y),\quad x,y\in C.
\]
Associativity, (anti-)commutativity, and Jacobi identity for conformal algebras turn into 
very natural expressions in terms of the pseudo-product (see \cite{BDK2001}). For an arbitrary 
variety Var of algebras, this approach leads to the notion of a Var-conformal algebra~\cite{Kol2006CA}.
In particular, for the variety of Poisson algebras, we obtain the following

\begin{definition}\label{defn:PoissonConformal}
A {\em Poisson conformal algebra} $P$ is an $H$-module equipped with two 
$\lambda$-products 
\[
 (x\ool\lambda y), [x\ool\lambda y] \in P[\lambda ], \quad x,y\in P, 
\]
such that 
\eqref{eq:3/2-linear} holds for both $\lambda$-products, 
$(x\ool{\lambda }y)$ is associative and commutative, 
$[x\ool{\lambda} y]$ is anti-commutative and satisfies the Jacobi identity \eqref{eq:ConfJacobi}, 
and the following {\em conformal Leibniz rule} holds:
\begin{equation}\label{eq:ConfLeibniz}
 [x\ool{\lambda } (y\ool{\mu } z) ] = ([x\ool{\lambda } y]\ool{\lambda + \mu }z )  + (y\ool{\mu } [x\ool{\lambda } z]),
 \quad x,y,z\in P.
\end{equation}
\end{definition}

\begin{remark}\label{rem:Leibniz2}
Relation \eqref{eq:ConfLeibniz} is equivalent to 
\begin{equation}\label{eq:ConfLeibniz2}
 [(x\ool{\lambda } y)\ool{\mu } z ] = (y\ool{\mu -\lambda }[x\ool{\lambda } z] )  + (x\ool{\lambda } [y\ool{\mu-\lambda } z]).
\end{equation}
\end{remark}

\begin{remark}\label{rem:Leibniz1-1}
 Note that \eqref{eq:ConfLeibniz} holds on every associative conformal algebra $C$ relative to 
 $[x\ool{\lambda }y]$ given by  \eqref{eq:ConfCommutator}:
 \[
  [x\ool{\lambda }y] = (x\ool\lambda y) - (y\ool{-\partial-\lambda } x), \quad x,y\in C.
 \]
An equivalent form of \eqref{eq:ConfLeibniz2} in the absence of commutativity is 
\begin{equation}\label{eq:ConfLeibniz2As}
 [(x\ool{\lambda } y)\ool\mu z] = (x\ool\lambda [y\ool{\mu-\lambda } z]) + \{[x\ool\lambda z]\ool{\mu-\lambda} y\}.
\end{equation}
\end{remark}

Definition \ref{defn:PoissonConformal} seems close to the notion of a Poisson vertex algebra introduced in 
\cite{BDSK09}. However, it is not clear what is a formal relation between them.

\begin{example}\label{exmp:CurrP}
Let $V$ be an ordinary Poisson algebra. Then 
$P=H\otimes V$ equipped with operations
$(a\ool{\lambda }b) = ab$, $[a\ool{\lambda }b] = [a,b]$
for
$a,b\in V$
is a Poisson conformal algebra denoted $\Cur V$.
\end{example}

\begin{example}\label{exmp:WeylPoiss}
 Consider $PV_2=\Bbbk[\partial , v]\simeq H\otimes \Bbbk[v]$
 as a current associative commutative conformal algebra over $\Bbbk[v]$
 equipped with 
 $$
 [v^m\ool{\lambda } v^{n}] = (m\partial + (n+m)\lambda )v^{n+m-1}.
 $$
It is straightforward to check that $P$ is a Poisson conformal algebra.  
\end{example}

Example \ref{exmp:WeylPoiss} (as a Lie conformal algebra, it is a sort of 
Block-type Lie conformal algebra studied in \cite{SuBlock}) 
is a particular case of a more general structure.

\begin{proposition}\label{prop:PoisDerStructure}
 Given a Poisson algebra $V$ with a derivation $D$, the free $H$-module $P=H\otimes V$ 
 is a Poisson conformal algebra relative to the following $\lambda$-products:
 \[
  \begin{gathered}
    (x\ool\lambda y) = xy, \\
    [x\ool{\lambda } y ] = [x,y] + \partial (yD(x)) + \lambda D(xy),
  \end{gathered}
 \]
$x,y\in V$. 
\end{proposition}

\begin{proof}
Conformal Lie bracket $[\cdot \ool\lambda \cdot ]$ turns $P$ into a quadratic Lie conformal algebra
studied in \cite{Xu1999}. 
 It remains to check \eqref{eq:ConfLeibniz} or \eqref{eq:ConfLeibniz2} which is straightforward.
\end{proof}

Relation between differential Poisson algebras and conformal algebras leads to a curious 
structure of an ordinary Poisson algebra on the space of Laurent polynomials over a Poisson algebra.

\begin{corollary}\label{cor:CoeffXu}
Suppose $V$ is a Poisson algebra with a derivation $D$, $V[t,t^{-1}]$ is the commutative algebra of Laurent polynomials 
over $V$. Then
\begin{equation}\label{eq:PoissonCoeff}
 [at^n,bt^m] = [a,b]t^{n+m}+(na D(b) - m bD(a))t^{n+m-1}, \quad a,b\in V,
\end{equation}
is a Poisson bracket on $V[t,t^{-1}]$.
\end{corollary}

\begin{proof}
Relation \eqref{eq:PoissonCoeff} 
along with the ordinary commutative multiplication on $V[t,t^{-1}]$
represent the coefficient algebra structure on the Poisson conformal algebra 
$P=H\otimes V$ from Proposition~\ref{prop:PoisDerStructure}.
\end{proof}

Poisson conformal algebras, even the simplest ones from example \ref{exmp:CurrP}, have a 
natural relation to representations of Lie conformal algebras. 

\begin{proposition}\label{prop:Poisson-LieMod}
 Let $P$ be a Poisson conformal algebra.
 Suppose $L$ is a conformal subalgebra of the underlying Lie conformal algebra $P$ relative to $[\cdot\oo\lambda \cdot]$.
 Then $P$ is a conformal module over $L$ with respect to the following operation:
 \[
  \langle a\ool\lambda u \rangle = [a\ool\lambda u] + \lambda (a\ool\lambda u), \quad a\in L,\ u\in P.
 \]
\end{proposition}

\begin{proof}
 It remains to check the conformal Jacobi identity
\begin{equation}\label{eq:PoisJacobiMod}
\langle a\ool\lambda \langle b\ool\mu u \rangle \rangle - \langle b\ool\mu \langle a\ool\lambda u \rangle \rangle
= \langle [a\ool{\lambda } b]\ool{\lambda+\mu} u\rangle .
 \end{equation}
Indeed, 
\begin{multline}\nonumber
 \langle a\ool\lambda \langle b\ool\mu u \rangle \rangle
 =
 \langle a\ool\lambda [b\ool\mu u] +\mu (b\ool{\mu} u)  \rangle  \\
 =
 [ a\ool\lambda [b\ool\mu u]] +\mu [ a\ool\lambda (b\ool{\mu} u)]
  +
 \lambda (a\ool\lambda [b\ool\mu u]) +\lambda\mu (a\ool\lambda (b\ool{\mu} u)) \\
 =
 [ a\ool\lambda [b\ool\mu u]] 
 +\mu ([ a\ool\lambda b]\ool{\lambda + \mu} u) 
 + \mu (b\ool\mu [a\ool{\lambda } u])         \\
 +\lambda (a\ool\lambda [b\ool\mu u]) 
 +\lambda\mu (a\ool\lambda (b\ool{\mu} u)).
\end{multline}
Hence, the left-hand side of \eqref{eq:PoisJacobiMod} is equal to 
\begin{equation}\label{eq:PoisJacobiMod_LHS}
  [[a\ool{\lambda }b ]\ool{\lambda+\mu} u]
 + (\lambda +\mu)([ a\ool\lambda b]\ool{\lambda + \mu} u)
\end{equation}
since 
\[
 (a\ool\lambda (b\ool{\mu} u)) = (b\ool\mu (a\ool{\lambda } u))
\]
in every associative and commutative conformal algebra \cite{Roit2000},
\[
 ([ b\ool\mu a]\ool{\lambda + \mu} u) = - ([ a\ool{-\mu-\partial} b]\ool{\lambda + \mu} u) = 
 - ([ a\ool\lambda b]\ool{\lambda + \mu} u)
\]
by \eqref{eq:ConfAComm} and \eqref{eq:3/2-linear}.
Obviously, 
\eqref{eq:PoisJacobiMod_LHS} coincides with the right-hand side of 
\eqref{eq:PoisJacobiMod}.
\end{proof}

\begin{corollary}\label{cor:PoissonCurrent}
 If $V$ is an ordinary Poisson algebra then 
 $P=\Cur V$ is a conformal module over $\Cur \mathfrak g$,
 where $\mathfrak g$ is a Lie subalgebra of $P$.
\end{corollary}

The purpose of this note is to establish more complicated Poisson conformal algebras whose commutative 
operation may not be reduced to a current-type structure. As in the case of ordinary algebras, it is natural
to seek among universal enveloping associative algebras. 

Given a Lie algebra $\mathfrak g$, let $P(\mathfrak g)$ be its symmetric algebra 
equipped with Poisson bracket $[\cdot,\cdot]$ induced by the commutator on $\mathfrak g$. 
As a linear space, $P(\mathfrak g)$ is isomorphic to the universal associative envelope $U(\mathfrak g)$
by the Poincar\'e--Birkhoff--Witt (PBW) Theorem.

For Lie conformal algebras, we have a hierarchy of universal associative envelopes \cite{Roit2000}.
Given a Lie conformal algebra $L$, an {\em associative envelope} of $L$ is an associative conformal algebra $C$
 equipped with a homomorphism (not necessarily injective) $\varphi : L\to C^{(-)}$ such that $C$ is generated 
by $\varphi(L)$ as a conformal algebra. 
Suppose $X$ is a generating set of $L$ as of $H$-module.
Fix a function $N:X\times X\to \mathbb Z_+$. Then the class of associative envelopes $(C,\varphi)$ of $L$, 
such that
$N_C(\varphi(x),\varphi(y))\le N(x,y)$ for all $x,y\in X$
contains a unique (up to isomorphism)
universal associative envelope $(U(L; X,N),\iota )$, $\iota :L\to U(L; X,N)^{(-)}$.

Associative conformal algebra $U(L;X,N)$ has a natural ascending filtration, the corresponding 
associated graded space carries a structure of a Poisson conformal algebra (see Section~\ref{sec:Brackets} for details).
In order to study this structure, we need to determine a normal form of elements in $U(L;X,N)$.
As in the case of ordinary algebras, $U(L;X,N)$ is determined by defining relations.
In the next section, we present a general approach to the study of conformal algebras
given by generators and relations, a sort of Composition-Diamond Lemma (CD-Lemma) for conformal algebras. 
Previous versions of the CD-Lemma for associative conformal algebras \cite{BFK00, BFK2004, NiChen2017}
work for bounded functions~$N$. Our approach does not depend on $N$ and, which is more important, 
we reduce the problem to modules over ordinary associative algebras. Therefore, one may apply 
available computer algebra packages for computations in conformal algebras within this approach.

\section{A version of the Diamond Lemma for associative conformal algebras}\label{sec:GSB}

Let $X$ be a well-ordered set, and let $N:X\times X\to \mathbb Z_+$ be a 
fixed function. Denote by $\Conf(X,N)$ the free associative conformal algebra 
generated by $X$ with respect to locality function $N$ \cite{Roit1999}. 
One may choose a linear basis of $\Conf(X,N)$ in the form 
\begin{equation}\label{eq:Conf_basis}
\begin{gathered}
 \partial ^s (a_1\oo{n_1} (a_2\oo{n_2} \dots \oo{n_{k-1}}(a_k \oo{n_k} a_{k+1} ) \dots )), \\
  k,s\in \mathbb Z_+,\ a_i\in X,\ 0\le n_i<N(a_i,a_{i+1}).
\end{gathered}
\end{equation}
Consider linear operators $L_n^a$ and $R_n^a$ on $\Conf(X,N)$ defined as follows:
\[
 L_n^a(f) = a\oo{n} f,\quad R_n^a(f) = \{f\oo{n} a\}, 
\]
$a\in X$, $n\in \mathbb Z_+$, $f\in \Conf(X,N)$,
where 
$\{x\oo{n} y\}$ is given by \eqref{eq:Curly}.
The axioms of an associative conformal algebra 
imply the following relations to hold in $\End\Conf(X,N)$:
\begin{gather}
 L_n^a\partial = \partial L_n^a + nL_{n-1}^a,  \label{eq:Alg_A:Relations1} \\
 R_n^a\partial = \partial R_n^a + nR_{n-1}^a,  \label{eq:Alg_A:Relations2} \\
 R_m^bL_n^a = L_n^aR_m^b.                      \label{eq:Alg_A:Relations3}
\end{gather}
Hence, $\Conf(X,N)$ is a (left) module over the ordinary associative algebra $A(X)$
generated by formal variables $\partial $, $L_n^a$, $R_n^a$ relative to 
the relations \eqref{eq:Alg_A:Relations1}--\eqref{eq:Alg_A:Relations3}.

It is not hard to find the defining relations of $\Conf(X,N)$ as of $A(X)$-module. 

\begin{theorem}[\cite{Kol_ICAC17}]
Let $M$ be a left $A(X)$-module generated by $X$ relative to 
the defining relations
 \begin{gather}
 L_n^a b =0, \quad a,b\in X,\ n\ge N(a,b),                       \label{eq:Mod_M:Locality}\\
 R_n^b a = (-1)^n\sum\limits_{s=0}^{N(a,b)-n-1} \partial^{(s)} L_{n+s}^a b,
    \quad a,b\in X,\ n\in \mathbb Z_+.                        \label{eq:Mod_M:RightMul}
\end{gather}
Then $M$ is isomorphic to $\Conf(X,N)$ as $A(X)$-module.
\end{theorem}

\begin{proof}[Sketch of the proof]
Obviously, \eqref{eq:Mod_M:Locality} and \eqref{eq:Mod_M:RightMul} hold in $\Conf(X,N)$.
The only problem is to show that the $A(X)$-module homomorphism 
$M\to \Conf(X,N)$ is injective. To resolve this problem, 
it is natural to apply the Gr\"obner--Shirshov bases technique for 
modules \cite{KangLee}.

Given a well order on $X$, extend it to $L_n^a$ and $R_n^a$ by the natural rule
$L_n^a <L_m^b$ (or $R_n^a<R_m^b$) if $n<m$ or $n=m$ and $a<b$; 
assume $\partial <L_n^a<R_m^b$ for all $a,b\in X$, $n,m\in \mathbb Z_+$. 
Next, define the following monomial order $\prec $ on the words in the alphabet $\partial $, $L_n^a$, $R_n^a$:
compare two words first by their degree in the variables $R_n^a$, then by deg-lex order.

Obviously, relations \eqref{eq:Alg_A:Relations1}--\eqref{eq:Alg_A:Relations3} form a GSB
 of $A(X)$, which is actually the universal enveloping algebra of some Lie algebra. 
Hence, the linear basis $B$ of $A(X)$ consists of all words of the form
\[
 \partial^s L_{n_1}^{a_1}\dots L_{n_k}^{a_k} R_{m_1}^{b_1}\dots R_{m_t}^{b_t}.
\]

Expand the above monomial order $\prec $ to the monomials in the free $A(X)$-module generated by $X$: 
for $u,v\in B$ and $x,y\in X$, 
let $ux\prec vy$ if and only if $u\prec v$ or $u=v$   and $x<y$.

It is easy to see that 
\eqref{eq:Alg_A:Relations3}, \eqref{eq:Mod_M:Locality}, and \eqref{eq:Mod_M:RightMul}
imply 
a series of relations 
\begin{equation}\label{eq:Mod_M:Locality_ex}
L_n^aL_m^bu = \sum\limits_{q\ge 1} (-1)^{q+1} \binom{n}{q} L_{n-q}^aL_{m+q}^b u,
\end{equation}
where $a,b\in X$, $n\ge N(a,b)$, $m\in \mathbb Z_+$, 
and $u$ is of the form 
\[
 L_{m_1}^{c_1}L_{m_2}^{c_2}\dots L_{m_k}^{c_k} c_{k+1} ,\quad  
c_i\in X,\ k,m_i\in \mathbb Z_+.
\]

Consider the {\em reduced words}, i.e., those monomials in the free $A(X)$-module generated by $X$ that do not 
contain a subword equal to a principal part of 
\eqref{eq:Mod_M:Locality}--\eqref{eq:Mod_M:Locality_ex}, 
The latter principal parts are equal to 
$L_n^a b$ for $n\ge N(a,b)$, $R_n^ab$, and $L_n^aL_m^bu$ for $n\ge N(a,b)$.
Therefore, $M$ is spanned by the reduced words that are of the form 
\[
 \partial^s L^{a_1}_{n_1} L^{a_2}_{n_2} \dots L^{a_k}_{{n_k}} a_{k+1},
\]
and their images in \eqref{eq:Conf_basis} are linearly independent. Hence, 
\eqref{eq:Mod_M:Locality}--\eqref{eq:Mod_M:Locality_ex} is a GSB of 
$\Conf(X,N)$ in the sense of \cite{KangLee}.
\end{proof}

By the definition of an $A(X)$-module structure on $\Conf(X,N)$, 
a subspace $I\subset\Conf(X,N)$ 
is an ideal of the conformal algebra $\Conf(X,N)$
if and only if $I$ is an $A(X)$-submodule. 
If $S\subset \Conf(X,N)$ is a set of conformal polynomials then 
the ideal generated by $S$ in the conformal algebra $\Conf(X,N)$ coincides 
with the $A(X)$-submodule generated by~$S$. Therefore, 
in order to solve the word problem in an associative conformal algebra 
defined by generators and relations  it is enough to solve 
that problem in the corresponding module over an ordinary associative algebra. 

In general, if an associative algebra $A$ and (left) $A$-module $M$ are defined via generators and relations
(say, $A$ and $M$ are generated by $X$ and $Y$, respectively)
then the problem of finding normal forms in $M$ was considered in \cite{KangLee}.
However, one may apply the ordinary Composition-Diamond Lemma for associative algebras to the split null extension 
$A\oplus M$ assuming obvious additional relations $yx=0$, $yz=0$ for $x\in X$, $y,z\in Y$.

\begin{corollary}[CD-Lemma]\label{cor:CD-lem}
 Let $S$ be a set of conformal polynomials in $\Conf(X,N)$ considered 
 as elements of the free $A(X)$-module generated by $X$. Then the following 
 conditions are equivalent:
 \begin{enumerate}
  \item $S$ together with \eqref{eq:Mod_M:Locality}, \eqref{eq:Mod_M:RightMul}, and \eqref{eq:Mod_M:Locality_ex} is a GSB of an $A(X)$-module;
  \item $S$-reduced words of the form 
  \[
  \partial ^s L^{a_1}_{n_1}L^{a_2}_{n_2} \dots L^{a_k}_{n_k} a_{k+1} , \quad
  k,s\in \mathbb Z_+,\ a_i\in X,\ 0\le n_i<N(a_i,a_{i+1}),
  \]
  form a linear basis of $\Conf(X,N \mid S)$.
 \end{enumerate}
\end{corollary}

To study the structure of a universal associative enveloping conformal algebra 
of a Lie conformal (super)algebra, it is convenient to add more defining relations 
to the algebra $A(X)$. 
Namely, 
suppose $L$ is a Lie conformal superalgebra which is a free $H$-module, and let 
$X$ be a homogeneous basis of $L$ over $H$. 
Recall that the following identity holds in every associative conformal algebra \cite{Roit2000}:
\[
 x\oo{n} (y\oo{m} z) - (-1)^{|x||y|}y\oo{m} (x\oo{n} z) = \sum\limits_{s\ge 0} \binom{n}{s} [x\oo{s} y]\oo{n+m-s} z.
\]
Therefore, $U(L; X,N)$ is a module over the associative algebra 
$A(X,L)$ generated by $\partial $, $L_n^a$, $R_n^a$ ($a\in X$, $n\in \mathbb Z_+$) 
relative to the defining relations 
\eqref{eq:Alg_A:Relations1}--\eqref{eq:Alg_A:Relations3} and
\begin{equation}\label{eq:CommL-operators}
\begin{gathered}
 L^a_nL^b_m  - (-1)^{|a||b|}L^b_mL_n^a  = \sum\limits_{s\ge 0}\binom{n}{s} L^{[a\oo{s} b]}_{n+m-s} ,\\
 a,b\in X,\ n,m\in \mathbb Z_+,\ L_n^a>L_m^b.
\end{gathered}
\end{equation}
Here we assume $L^{\partial x}_n = -nL_{n-1}^x$
to express the right-hand side of \eqref{eq:CommL-operators}.

Defining relations of $U(L; X,N)$
as of $A(X,L)$-module include \eqref{eq:Mod_M:Locality}, \eqref{eq:Mod_M:RightMul}, and 
\begin{equation}\label{eq:Mod_M:Comm}
 R_n^ab - (-1)^{|a||b|} L_n^ab = -[a\oo{n} b],\quad a,b\in X,\ n\in \mathbb Z_+. 
\end{equation}

It is not hard to see that \eqref{eq:Alg_A:Relations1}--\eqref{eq:Alg_A:Relations3}, \eqref{eq:CommL-operators}
form a GSB of the associative algebra $A(X,L)$. In order to determine the 
structure of $U(L;X,N)$ it is enough to find a GSB of the $A(X,L)$-module 
generated by $X$ relative to \eqref{eq:Mod_M:Locality}, \eqref{eq:Mod_M:RightMul}, and \eqref{eq:Mod_M:Comm}.

\section{Example: Universal envelope of the Neveu--Schwartz conformal superalgebra}\label{sec:exmp}

Consider $L=K_1$, the Neveu--Schwartz conformal superalgebra (see \cite{Kac98}).
Then $X=\{v,g\}$, $|v|=0$, $|g|=1$, and the multiplication table is given by
\[
 [v\ool{\lambda } v] = \partial v + 2\lambda v, \quad
 [g\ool{\lambda } v] =  \dfrac{1}{2} \partial g + \dfrac{3}{2}\lambda g, \quad 
 [g\ool{\lambda } g] = -\dfrac{1}{2} v.
\]
Assume $v<g$.
For convenience of computation, let us slightly change the order $\prec $ assuming 
$L_0^x < L_1^x <\partial <L_2^x <\dots $, $x\in X$ (other rules remain the same).
The set of defining relations of $A(X,K_1)$ consists of
\begin{equation}\label{eq:K1_Alg_Relations}
\begin{gathered}
\partial L_0^x = L_0^x\partial ,\quad \partial L_1^x=  L_1^x\partial - L_0^x, \quad x\in X; \\
L_n^x \partial  = \partial L_n^x + nL_{n-1}^x,\quad n\ge 2,\ x\in X; \\
R_n^x \partial  = \partial R_n^x + nR_{n-1}^x,\quad n\ge 0,\ x\in X; \\
R_n^xL_m^y = L_m^yR_n^x,\quad n,m\ge 0, \ x,y\in X; \\
L_n^v L_m^v = L_m^vL_n^v + (n-m)L_{n+m-1}^v,\quad n>m\ge 0; \\
L_n^v L_m^g = L_m^g L_n^v + \left(\frac{1}{2}n - m\right)L^g_{n+m-1}, \quad  n>m\ge 0; \\
L_n^g L_m^v = L_m^v L_n^g + \left(n - \frac{1}{2} m\right)L^g_{n+m-1}, \quad n\ge m\ge 0; \\
L_n^g L_m^g = -L_m^g L_n^g -\dfrac{1}{2} L^v_{n+m}, \quad n>m\ge 0;\\
L_n^gL_n^g = -\dfrac{1}{4} L^v_{2n},\quad n\ge 0.
\end{gathered}
\end{equation}

In \cite{Kol2002}, a GSB of $U(K_1; X,N)$ in the sense of \cite{BFK2000}
was found for 
\begin{equation}\label{eq:K1-Locality}
N(v,v)=3,\quad N(v,g)=N(g,v)=N(g,g)=2. 
\end{equation}
Let us show how the technique exposed in the 
previous section works for the same locality function~$N$.

According to the general scheme described above, 
the following relations determine $U(K_1;X,N)$ as an $A(X,K_1)$-module:
\begin{equation}\label{eq:K1_Mod_Relations}
 \begin{gathered}
L_n^xy = 0,\quad x,y \in X,\ n\ge N(x,y); \\
R_n^xy = 0, \quad x,y\in X,\ n\ge N(y,x); \\
R_0^v v = L_0^v v -\partial L_{1}^v v + \partial^{(2)} L_{2}^v v; \\
R_1^v v = - L_1^v v+\partial L_2^v v, \quad R_2^v v = L_2^v v; \\
R_n^x y = (-1)^n (L_n^y x - \partial L_{n+1}^y x),\quad
\{x,y\} =X, \ n=0,1; \\
R_0^g g = L_0^g g - \partial L_1^g g; \quad
R_1^g g = -L_1^g g; \\
R_0^v v = L_0^v v -\partial v; \quad
R_1^v v = L_1^v v - 2v;  \\
R_0^v g = L_0^v g - \partial g; \quad
R_1^v g = L_1^v g - \dfrac{3}{2} g;  \\
R_0^g v = L_0^g v - \dfrac{1}{2}\partial g; \quad
R_1^g v = L_1^g v - \dfrac{3}{2} g;\\
R_0^g g = -L_0^g g - \dfrac{1}{2} v;\quad 
R_1^g g = -L_1^g g .  \\
 \end{gathered}
\end{equation}

Calculation of a GSB of the $A(X,K_1)$-module 
generated by $X$ with defining relations 
\eqref{eq:K1_Mod_Relations} is a standard computational task: one has to add 
all non-trivial compositions to the set of defining relations.

\begin{theorem}\label{thm:K1_GSB}
In order to obtain a GSB of $U(K_1;X,N)$ 
for $N$ given by \eqref{eq:K1-Locality}
it is enough to enrich the system \eqref{eq:K1_Mod_Relations} with 
the following relations:
\begin{equation}\label{eq:K1_GSB}
 \begin{gathered}
L_2^v v = -2 L_1^g g; \quad 
L_1^v v = -2 L_0^g g  + \dfrac{1}{2}v; \\
L_1^g v = - L_1^v g  + \dfrac{3}{2}g; \quad
L_1^v L_1^v g = \dfrac{3}{2}L_1^v g  -\dfrac{1}{2}g; \\
L_1^v L_1^g g = \dfrac{1}{2} L_1^g g; \quad
L_0^g L_1^v g = \dfrac{1}{2} L_0^v L_1^g g + \dfrac{1}{2}L_0^g g; \\
L_0^g L_1^g g = -\dfrac{1}{2}L_1^v g  + \dfrac{1}{4}g; \\
L_1^v \partial^s v = -2L_0^g\partial ^s g + \dfrac{1}{2} \partial^s v + sL_0^v\partial^{s-1} v,\quad s\ge1; \\
L_1^g \partial^s v = -L_0^v\partial ^{s-1} g + \partial^s g + (s+1)L_0^g\partial^{s-1} v,\quad s\ge1; \\
L_1^v \partial^s g = -L_0^g\partial ^{s-1} v + \dfrac{1}{2}\partial^s g + (s+1)L_0^v\partial^{s-1} g,\quad s\ge1; \\
L_1^g \partial^s g = (s+2)L_0^g \partial^{s-1} g + \dfrac{1}{2}\partial^{s-1} v,\quad s\ge1.
 \end{gathered}
\end{equation}
\end{theorem}

This result is agreed with the computations in \cite{Kol2002}, so we do not present the details here. 
However, Theorem \ref{thm:K1_GSB} may be easily checked by means of computer algebra 
systems 
providing an opportunity of (step-by-step) computation of GSBs 
in non-commutative associative algebras. 

\begin{corollary}\label{cor:K1_basis}
The following words form a linear basis of $U(K_1;X,N)$: 
\begin{equation}\label{eq:K1_reduced}
    \big(L_0^v\big )^n \partial^s x, \quad
   \big(L_0^v\big )^n L_0^g \partial^s x, \quad
   \big(L_0^v\big )^n L_1^x g, \quad n\ge 0,\ x\in X. 
 \end{equation}
\end{corollary}

\begin{proof}
 Let $S$ be the set of defining relations  
 \eqref{eq:K1_Alg_Relations}, \eqref{eq:K1_Mod_Relations}, \eqref{eq:K1_GSB}.
 Then \eqref{eq:K1_reduced} is exactly the set of $S$-reduced words in the free $A(X,K_1)$-module 
 generated by~$X$. 
\end{proof}

 To make sure that the results of Theorem~\ref{thm:K1_GSB} and Corollary~\ref{cor:K1_basis} 
 are correct,
 one may recall the following presentation of $K_1$. Consider the associative conformal 
 algebra $\Cend_2$ with the natural $\mathbb Z_2$-grading as $\Bbbk[x,\partial ]\otimes M_{1|1}(\Bbbk )$.
 Then 
 \[
  v = \begin{pmatrix}
        x & 0 \\ 0 & x-\frac{1}{2}\partial 
      \end{pmatrix},
\quad       
  g = \frac{1}{2}\begin{pmatrix}
        0 & x \\ -1 & 0 
      \end{pmatrix}
 \]
span a Lie conformal superalgebra $L$ in $\Cend_2^{(-)}$ isomorphic to $K_1$.
Associative envelope $C$ of $L$ in $\Cend_2$ coincides with the set of all 
matrices 
\begin{equation}\label{eq:Cend2Q}
 \begin{pmatrix}
  xf_{11}(x,\partial ) & xf_{12}(x,\partial ) \\
   f_{21} (x,\partial ) & f_{22}(x,\partial )
 \end{pmatrix},
 \quad f_{ij}\in \Bbbk[x,\partial ].
\end{equation}
Straightforward computation shows that the images of \eqref{eq:K1_reduced} in $\Cend_2$
exactly form a linear basis of~$C$.

\begin{corollary}\label{cor:Vir3Envelope}
 For the Virasoro conformal algebra $\Vir$, 
 a linear basis of $U(\Vir; v,3)$ consists of 
 the words 
 \begin{equation}\label{eq:Vir-3Basis}
  L_0^n \partial^s L_1^m v,\quad L_0^nL_1^mL_2v,\quad n,m,s\ge 0,
 \end{equation}
where $L_k^n$ stands for $(L_k^v)^n$. In particular, 
$U(L;v,3)$ is a free $H$-module generated by 
\[
x_{n+1}= L_0^nv,\quad y_{n+1,m+1} = L_0^nL_1^mL_2v,\quad n,m\ge 0.
\]
\end{corollary}

\begin{proof}
To find a GSB of $U(\Vir; v,3)$ it is enough to add the following to the initial set of defining relations:
\[
 L_2L_2v = 0,\quad \partial L_2v - 2L_1v = 0, 
\]
see \cite{Kol_ICAC17} for details. The set of reduced words coincides with \eqref{eq:Vir-3Basis}.
\end{proof}

\section{Conformal Poisson brackets on associative envelopes of the Virasoro conformal algebra}\label{sec:Brackets}

Let $L$ be a Lie conformal algebra generated by a set $X$ as an $H$-module. For a fixed function $N:X\times X\to \mathbb Z_+$, 
consider the universal associative conformal envelope $U=U(L;X,N)$. The latter is a homomorphic image of $\Conf(X,N)$, 
so there is an ascending filtration 
\[
 U = \bigcup\limits_{n\ge 1}F_nU, \quad (F_nU\ool\lambda F_mU)\subseteq F_{n+m}U[\lambda ], 
\]
where $F_{n}U$ consists of images of all words \eqref{eq:Conf_basis} of degree $k+1\le n$ in $X$.
Assume $F_0U=\{0\}$.

Consider the associated graded linear space 
\[
 \gr U(L;X,N) = \bigoplus _{n\ge 1} F_nU/F_{n-1}U
\]
equipped with well-defined operations
\[
 \partial \bar u = \overline{\partial u},\quad u\in F_nU,
\]
and 
\[
 (\bar u \bol\lambda \bar v) = \overline {(u\ool\lambda v)}\in F_{n+m}U/F_{n+m-1}U, \quad u\in F_nU,\ v\in F_mU.
\]
The associative and commutative conformal algebra obtained is a Poisson conformal algebra relative to 
\begin{equation}\label{eq:UnivPoisson}
  [\bar u\bol\lambda \bar v ] = \overline{(u\ool\lambda v) - \{v\ool\lambda u\}} \in F_{n+m-1}U/F_{n+m-2}U 
\end{equation}
for $u\in F_nU$, $v\in F_mU$.
The operation \eqref{eq:UnivPoisson} is well-defined since 
 \eqref{eq:ConfLeibniz} and \eqref{eq:ConfLeibniz2} imply
\[
 [u\ool\lambda v] = (u\ool\lambda v) - \{v\ool\lambda u\} \in F_{n+m-1}U.
\]
It follows from the same relations that \eqref{eq:ConfAComm}, \eqref{eq:ConfJacobi}, and \eqref{eq:ConfLeibniz} hold
for the operations $[\cdot \bol\lambda \cdot]$ and  $(\cdot \bol\lambda \cdot)$ on $\gr U(L;X,N)$.
Therefore, $\gr U(L;X,N)$ carries a natural structure of a Poisson conformal algebra, let us denote it by $P(L;X,N)$.

Obviously, $\mathbb Z_2$-graded version of this construction leads to a Poisson conformal superalgebra structure 
on the associated graded universal associative conformal envelope of a Lie conformal superalgebra.

\begin{example}\label{exmp:grUVir2}
 For the Virasoro Lie conformal algebra, 
 $P(\Vir; \{v\}, 2)$ is isomorphic to the Poisson conformal algebra $PV_2$
 from Example~\ref{exmp:WeylPoiss}. 
\end{example} 
 
It is easy to find a GSB of $U=U(\Vir; \{v\}, 2)$ (see \cite{BFK2000}), 
  the corresponding set of reduced words 
 is $\partial^s (L_0^v)^n v \in F_{n+1}U$, $n,s\ge 0$. Since $L_1^vv = v$, we have 
 the isomorphism of conformal algebras $\gr U \simeq \Cur \Bbbk[v]$, $(L_0^v)^nv\mapsto v^{n+1}$.
 It is straightforward to evaluate conformal Poisson bracket using \eqref{eq:ConfLeibniz} and \eqref{eq:ConfLeibniz2}
 to get the formula from Example~\ref{exmp:WeylPoiss}.

\begin{example}\label{exmp:grUVir3}
On the 1-generated free commutative conformal algebra one may define 
a Poisson conformal bracket induced by the Virasoro $\lambda $-bracket \eqref{eq:VirProduct}.
Let us denote this Poisson algebra $PV_3$.
\end{example}

Corollary \ref{cor:Vir3Envelope} and \cite[Section 9.3]{BFK2000} show that
$\gr U(\Vir; \{v\},3)$ is isomorphic to 
the 1-generated commutative conformal algebra $\ComConf(\{v\},3)$. 
Let us evaluate, for example, $(x_n\bol\lambda x_m)$. 
By definition, 
\begin{multline}\nonumber
(x_n\bol\lambda x_m) = L_0^{n-1}v \bol\lambda L_0^{m-1}v  \\
= L_0^{n+m-2}(v\bol\lambda v ) = L_0^{n+m-2}(L_0v + \lambda L_1v +\lambda^{(2)}L_2v) \\
= L_0^{n+m-1}v + \frac{1}{2}\lambda\partial L_0^{n+m-2} L_1v + \lambda^{(2)}L_0^{n+m-2}L_2v \\
= x_{n+m} + \frac{1}{2} (\lambda\partial + \lambda^2)y_{n+m-1,1}.
\end{multline}
Similarly,
\[
\begin{gathered}
 (x_n\bol\lambda y_{m,k}) = y_{n+m,k} + \lambda y_{n+m-1,k+1}, \\
 (y_{n,m}\bol\lambda y_{k,l}) = 0, 
\end{gathered} 
\]
Explicit formulas for the Poisson conformal bracket on $PV_3$ may be deduced from 
\eqref{eq:ConfLeibniz} and \eqref{eq:ConfLeibniz2}.
For example, 
\begin{multline}{}\nonumber
[x_1\bol\lambda x_m] = [x_1\bol\lambda L_0x_{m-1}] = ([v\bol\lambda v]\bol\lambda x_{m-1}) + L_0[x_1\bol\lambda x_{m-1}] \\
= \lambda (x_{m} + \frac{1}{2} (\lambda\partial + \lambda^2)y_{m-1,1}) + L_0[x_1\bol\lambda x_{m-1}] \\
= (m-1)\lambda x_m + \frac{m-1}{2} \lambda^2 (\lambda+\partial ) y_{m-1,1} + L_0^{m-1}(\partial +2\lambda )v \\
= (\partial+(m+1)\lambda )x_{m} + \frac{m-1}{2}\lambda^2(\lambda + \partial )y_{m-1,1}.
\end{multline}
In a similar way, 
\[
[x_n\bol\lambda x_m] = (n\partial + (n+m)\lambda )x_{n+m-1} + \dfrac{1}{2}\lambda (\partial + \lambda ) ((n-1)\partial + (n+m-2)\lambda )y_{n+m-2,1}. 
\]
To compute $[x_n\bol\lambda y_{m,k}]$, let us start with $[x_1\bol\lambda y_{1,1}]=[v\bol\lambda L_2v]$ which is equal to 
the coefficient of $[v\bol\lambda (v\bol\mu v)]$ at $\mu^{(2)}$:
\begin{multline}{}\nonumber
 [v\bol\lambda (v\bol\mu v)] = ((\partial +2\lambda)v\bol{\lambda+\mu} v) + (v\bol\mu (\partial+2\lambda )v) \\
 =(\lambda-\mu)(L_0v+(\lambda+\mu)L_1v + (\lambda+\mu)^{(2)}L_2v) + (2\lambda + \partial+\mu)(L_0v + \mu L_1v + \mu^{(2)}L_2v) \\
 = (\lambda + \partial)\mu^{(2)} L_2v + \dots .
\end{multline}
Hence, 
\[
 [x_1\bol\lambda y_{1,1}] = (\lambda + \partial)y_{1,1}.
\]
In a similar way, 
\[
 [x_1\bol\lambda y_{m,1}] = (m\lambda +\partial )y_{m,1} + (m-1)\lambda^2 y_{m-1,2}, \quad m\ge 2.
\]
For $k\ge 2$, we may represent 
$y_{m,k} = L_1y_{m,k-1}$ and compute
\[
 [x_1\bol\lambda y_{m,k}] = - y_{m+1,k-1} + L_1[v\bol\lambda y_{m,k-1}].
\]
Therefore, 
\begin{multline}\nonumber 
[x_1\bol\lambda y_{m,k}] = (1-k)y_{m+1,k-1} + L_1^{k-1}[x_1\bol\lambda y_{m,1}]  \\
= (\partial+\lambda)y_{m,k} + (m-1)\lambda^2 y_{m-1,k+1} - (k-1)y_{m+1,k-1}.
\end{multline}
Finally, 
\begin{multline}\nonumber 
 [x_n\bol\lambda y_{m,k}] = n(1-k)y_{n+m,k-1} + (n(\partial+\lambda) - (n-1)(k-1)\lambda )y_{n+m-1,k} \\
 + \lambda ((n+m-2)\lambda + (n-1)\partial )y_{n+m-2, k+1}
\end{multline}
by induction in $n\ge 1$.
It remains to note that 
\[
 [y_{n,m}\bol\lambda y_{k,l}] = 0. 
\]
Therefore, $PV_3$ is a central extension of $PV_2$ by means of the conformal module spanned by $y_{n,m}$, $n,m\ge 1$.

An interesting example of a Poisson conformal envelope of the Virasoro 
conformal algebra appears from the associative envelope of the Neveu--Schwartz conformal
superalgebra $K_1$.

\begin{example}\label{exmp:grUK_1}
 Suppose $L=K_1$ is the Neveu--Schwartz conformal superalgebra generated by $X=\{v,g\}$. 
 Then $U=U(K_1; X,N)$ for $N$ given by \eqref{eq:K1-Locality} 
 is isomorphic to the conformal subalgebra of $\Cend_2$
 that consists of matrices  \eqref{eq:Cend2Q}.
 Although $P=\gr U$ is not isomorphic to the supercommutative conformal algebra 
 generated by $X$ relative to the locality function $N$, the conformal Lie bracket on $K_1$
 induces Poisson conformal superalgebra structure on $P$ denoted $PK_1$. 
\end{example}

 According to Corollary~\ref{cor:K1_basis}, every $F_nU/F_{n-1}U\subset PK_1$, $n>1$, 
 is a 4-dimensional free $H$-module with a basis
 \[
  \bar a_n = a_n+F_{n-1}U,\quad \bar b_n=b_n+F_{n-1}U,\quad \bar e_n=e_n+F_{n-1}U,\quad \bar f_n=f_n+F_{n-1}U,
 \]
where 
\[
 a_n = \begin{pmatrix}
        x^n & 0 \\ 0 & x^n-\frac{1}{2}\partial x^{n-1}
       \end{pmatrix},
\quad 
b_n = \begin{pmatrix}
       0 & 0 \\ 0 & x^{n-2}
       \end{pmatrix},
\]
\[
e_n = \begin{pmatrix}
       0 & x^n \\ -x^{n-1} & 0
       \end{pmatrix},
\quad 
f_n = \begin{pmatrix}
       0 & 0 \\ x^{n-2} & 0
       \end{pmatrix}.
\]
Here $a_n$ and $b_n$ are even elements of the $\mathbb Z_2$-graded associative conformal algebra $U$, 
$e_n$ and $f_n$ are odd elements.
Let us evaluate explicitly the structure of the even part $PK_{10}$ of the Poisson conformal superalgebra $PK_1$.

By the definition of $\Cend_2$, 
\[
\begin{aligned}{}
 [a_n\ool\lambda a_m] & = a_n(-\lambda, x)a_m(\partial+\lambda, x+\lambda ) - a_m(\lambda+\partial, x)a_n(-\lambda, x-\partial-\lambda) \\
 &=\begin{pmatrix}
    x f(\partial, \lambda , x) & 0 \\
    0 & g(\partial, \lambda , x)
   \end{pmatrix}.
\end{aligned}
\]
To find the component from $F_{n+m-1}U/F_{n+m-2}U$, find the principal (relative to $x$) term of 
\[
 xf(\partial, \lambda , x) = (m\lambda + n(\partial+\lambda))x^{n+m-1} + \dots 
\]
and of 
\[
 g(\partial, \lambda, x) - \left( x - \frac{1}{2}\partial \right)f(\partial,\lambda, x) = \dfrac{\lambda^2+\partial\lambda}{4}((m+n-2)\lambda + (n-1)\partial ) x^{n+m-3}+\dots.
\]
Therefore, 
\begin{equation}\label{eq:PK_1[aa]}
 [\bar a_n\bol\lambda \bar a_m] = (n\partial + (m+n)\lambda )\bar a_{n+m-1} + \dfrac{\lambda^2+\partial\lambda}{4}((m+n-2)\lambda + (n-1)\partial )\bar b_{n+m-1}.
\end{equation}
In a similar way, we may evaluate 
\begin{gather}
[\bar a_n\bol\lambda \bar b_m] = ((n+m-2)\lambda + n\partial ) \bar b_{n+m-1},
                                             \label{eq:PK_1[ab]} \\
[\bar b_n\bol\lambda \bar b_m] = 0, 
                                             \label{eq:PK_1[bb]} \\
(\bar a_n\bol\lambda \bar a_m) = \bar a_{n+m} + \dfrac{\lambda^2+\partial\lambda}{4} \bar b_{n+m},
                                             \label{eq:PK_1aa} \\
(\bar a_n\bol\lambda \bar b_m) = \bar b_{n+m}, 
                                             \label{eq:PK_1ab} \\            
(\bar b_n\bol\lambda \bar b_m) = 0. 
                                             \label{eq:PK_1bb} 
\end{gather}
Therefore, $PK_{10}$ as an $H$-module is generated by $\bar a_n$, $\bar b_m$, $n\ge 1$, $m\ge 2$, 
and the multiplication table is given by \eqref{eq:PK_1[aa]}--\eqref{eq:PK_1bb}. 
It is easy to see that $PK_{10}$ is a central extension of $PV_2$ via the submodule generated by $b_m$, $m\ge 2$. 
The extension is not split since the 1st component of the grading does not intersect with 
the $H$-submodule spanned by $\bar b_n$, $n\ge 2$. To simplify the multiplication table, let us introduce 
\[
 \hat a_1 = \bar a_1,\quad \hat a_n = \bar a_n -\frac{1}{8} \partial^2\bar b_n, \ n\ge 2.
\]
Then 
\[
\begin{gathered}
 (\hat a_1\bol\lambda \hat a_m) = \hat a_{m+1}+\frac{1}{8}\lambda^2\bar b_{m+1},\quad  m> 1,\\
 (\hat a_n\bol\lambda \hat a_m) = \hat a_{n+m},\quad n,m>1, \\
 [\hat a_1\bol\lambda \hat a_m] = (\partial + (m+1)\lambda )\hat a_m + \frac{1}{8}((m-1)\lambda^3 -\partial\lambda^2-\partial^2\lambda -\partial^3)\bar b_m, \quad m>1,\\
 [\hat a_n\bol\lambda \hat a_m] = (n\partial + (n+m)\lambda )\hat a_{n+m-1} - \frac{n}{8} (4\lambda\partial ^2+\lambda^2\partial )\bar b_{n+m-1}, 
\quad n,m>1. 
\end{gathered}
\]

\begin{remark}
 For $N>3$, the associated graded Poisson conformal algebra 
 $PV_N = \gr U(\Vir; \{v\}, N)$ would not be a null extension of $PV_2$. 
 However, it is easy to see that $PV_{N+1}$ is a null extension of $PV_N$. 
 It is interesting problem to find the corresponding conformal modules
 and cocycles. This problem is closely related with finding a linear basis 
 of the free commutative conformal algebra.
\end{remark}

\end{document}